\documentclass{amsart}

\usepackage{a4wide,amsmath,amssymb}
\usepackage[toc,page]{appendix}

\usepackage{url}

\usepackage{hyperref}

\usepackage{mathtools}

\usepackage{xcolor}

\newtheorem{thm}{Theorem}[section]

\theoremstyle{remark}
\newtheorem{rem}[thm]{Remark}

\theoremstyle{definition}

\newtheoremstyle{Claim}{}{}{\itshape}{}{\itshape\bfseries}{:}{ }{#1}
\theoremstyle{Claim}

\newcommand{\T}{{\mathbb{T}}}

\newcommand{\R}{\mathbb{R}}

\newcommand{\eps}{\varepsilon}

\makeatletter
\@namedef{subjclassname@2020}{\textup{2020} Mathematics Subject Classification}
\makeatother

\theoremstyle{plain}

\begin{document}

\title[]{Rate of convergence of the vanishing viscosity method for Hamilton-Jacobi equations with Neumann boundary conditions}

\author{Alessandro Goffi}
\address{Dipartimento di Matematica e Informatica ``Ulisse Dini'', Universit\`a degli Studi di Firenze, 
viale G. Morgagni 67/A, 50134 Firenze (Italy)}
\curraddr{}
\email{alessandro.goffi@unifi.it}

\thanks{The author is member of the Gruppo Nazionale per l'Analisi Matematica, la Probabilit\`a e le loro Applicazioni (GNAMPA) of the Istituto Nazionale di Alta Matematica (INdAM), and he was partially supported by the INdAM-GNAMPA project 2025. The author wishes to thank the anonymous referee for a careful review which meant a significant
improvement of the original version of the manuscript.
}

\date{\today}


\date{\today}

\begin{abstract}
We study the quantitative small noise limit in the $L^\infty$ norm of certain time-dependent Hamilton-Jacobi equations equipped with Neumann boundary conditions, depending on the regularity of the data and the geometric properties of the domain. We first provide a $\mathcal{O}(\sqrt{\eps})$ rate of convergence for Hamilton-Jacobi equations with locally Lipschitz Hamiltonians posed on convex domains of the Euclidean space. We then enhance this speed of convergence in the case of quadratic Hamiltonians proving one-side rates of order $\mathcal{O}(\eps)$ and $\mathcal{O}(\eps^\beta)$, $\beta\in(1/2,1)$. The results exploit recent $L^1$ contraction estimates for Fokker-Planck equations with bounded velocity fields on unbounded domains used to derive differential Harnack estimates for the corresponding Neumann heat flow.
\end{abstract}

\maketitle

\section{Introduction}

In this note we study the rate of convergence of the vanishing viscosity approximation for the first-order (backward) Hamilton-Jacobi equation equipped with homogeneous Neumann boundary condition
\begin{equation*}\label{HJintro}
\begin{cases}
-\partial_t u+H(Du(x,t))=f(x,t)&\text{ in }\Omega\times(0,T),\\
\partial_\nu u=0& \text{ on }\partial\Omega\times(0,T),\\
u(x,T)=u_T(x)&\text{ in }\Omega,
\end{cases}
\end{equation*}
under the main assumption that $\Omega\subset\R^n$ is \textit{unbounded}. Heuristically, this process amounts to study the speed of convergence of $v^\eps=\partial_\eps u_\eps$ (or, formally, the behavior of $\frac{u_{\eps+\eta}-u_\eps}{\eta}$ as $\eta\to0^+$), where, denoting by $\eps>0$ the (small) viscosity parameter, $u_\eps$ solves the viscous problem
\begin{equation*}\label{HJviscintro}
\begin{cases}
-\partial_t u_\eps-\eps\Delta u_\eps+H(Du_\eps(x,t))=f(x,t)&\text{ in }\Omega\times(0,T),\\
\partial_\nu u_\eps=0& \text{ on }\partial\Omega\times(0,T),\\
u_\eps(x,T)=u_T(x)&\text{ in }\Omega.
\end{cases}
\end{equation*}
We will show that in the case of solutions satisfying a priori Lipschitz bounds (independent of $\eps$) and $H\in W^{1,\infty}_{\mathrm{loc}}(\R^n)$, if $\Omega$ is convex we have
\[
\|u_\eps-u\|_{L^\infty(\Omega\times(0,T))}\leq C\sqrt{\eps T}.
\]
This bound is of perturbative nature, i.e. it is a consequence of the linear part of the PDE, while the nonlinearity plays a minor role. In addition, this quantitative bound of order $\mathcal{O}(\sqrt{\eps T})$ is optimal. In fact, it agrees with that of the vanishing viscosity limit of the heat equation. It is easy to see that for $H=f=0$ one can write via \cite[Lemma II.1.3]{EngelNagel}, see also \cite[Lemma 1.14]{BertozziMajda}, for $g(\cdot,t)=u_\eps(\cdot,T-t)$
\[
\|g(\cdot)-u_T\|_{L^\infty}=\|e^{t\eps\Delta}u_T-u_T\|_{L^\infty}\leq C\|u_T\|_{W^{1,\infty}}\sqrt{\eps T}.
\]
If, in addition, $H$ is quadratic and $\Delta u_\eps\leq c(t)\in L^1(0,T)$, we can boost the rate on both sides and get the quantitative estimate
\[
-C\eps^\beta\leq u_\eps-u\leq \eps\int_0^T c(t),\ \beta\in(1/2,1).
\]
The results are based on duality methods and make use of some $L^1$-contraction estimates for Fokker-Planck equations discovered recently in \cite{GoffiTralli} within the analysis of geometric estimates for the heat flow. For this reason, we need to add the \textit{geometric} assumption that $\Omega$ satisfies a (uniform) interior cone property, cf. \eqref{intcone}. We mention that our estimates are new even when $\Omega$ is bounded and convex: in this case we do not need $\Omega$ to satisfy the interior cone condition, see \cite{GoffiTralli}. \\

The research on quantitative estimates for the vanishing viscosity approximation of Hamilton-Jacobi equations has recently received an incresing interest, mostly in connection with the quantitative study of the convergence problem in Mean Field Control. Classical results provide $\mathcal{O}(\sqrt{\eps})$ rates for problems posed in $\R^n$ or $\T^n$ using maximum principle methods \cite{CL84}, while the more recent \cite{EvansARMA} provides the same result using integral techniques. This rate can be boosted through estimates in weaker $L^p$-norms, cf. \cite{CGM}, via nonlocal approximations \cite{DroniouImbert,Goffi}, when $H$ is uniformly or strictly convex or under appropriate smallness conditions \cite{CirantGoffiIndiana}.\\
The case of problems with boundary conditions is much less studied and carried out for smooth, convex and bounded domains. When $H$ is locally Lipschitz, the convergence problem of the vanishing viscosity method was addressed in \cite{L82book,LionsDuke,PerthameSanders} for stationary equations with zero-th order coefficients, but little is known for evolutionary Hamilton-Jacobi equations. In particular, we provide a time-dependent counterpart of the quantitative convergence result found in \cite[Theorem 2]{PerthameSanders}. Our main novelty with respect to \cite{PerthameSanders} is the possibility of allowing unbounded domains and the treatment of evolution problems. We also mention other results obtained in rather different settings: \cite{Tran2011} for Dirichlet problems in bounded domains, \cite{Son} for problems with state constraints and \cite{Jakobsen} for general Neumann-type boundary conditions and H\"older continuous solutions.

Our results provide new advances in at least two directions, allowing the possibility of \textit{unbounded convex sets} and the treatment of \textit{time-dependent} problems: for us the result in the stationary case will be a byproduct, as in \cite{Tran2011} and \cite[Section 4.3.1]{CGM}. This note emphasizes that for problems with boundary conditions the geometry of the domain, together with the regularity of the data and the geometric assumptions of $H$, play an important role in the study of the rate of convergence of the vanishing viscosity approximation.\\
\par\smallskip

\textit{Outline}. Section \ref{sec;prel} collects some preliminary results for the Cauchy-Neumann problem of viscous Fokker-Planck equations. Section \ref{sec;rate} provides the main results of the paper: the first is a rate of convergence when the Hamiltonian is locally Lipschitz, while the second improved rate is obtained under the stronger semi-superharmonic condition on the solution $u_\eps$. Section \ref{sec;bounds} provides the proofs of these latter second order bounds under certain conditions on the Hamiltonian.

\section{Preliminary results on Fokker-Planck equations with Neumann boundary conditions}\label{sec;prel}
In this section we discuss well-posedness results and $L^1$ conservation properties of solutions for the forward Cauchy-Neumann problem solved by the function $\rho_\eps:=\rho_\eps(x,t)$
\begin{equation}\label{fp}
\begin{cases}
\partial_t\rho_\eps-\eps\Delta\rho_\eps+\mathrm{div}(b(x,t)\rho_\eps)=0&\text{ in }Q_\tau=\Omega\times(\tau,T),\\
\eps \partial_\nu \rho_\eps-\rho_\eps b(x,t)\cdot \nu=0&\text{ on }\Sigma_\tau:=\partial\Omega\times(\tau,T),\\
\rho_\eps(x,\tau)=\rho_\tau(x)&\text{ in }\Omega,
\end{cases}
\end{equation}
 where $\rho_\tau\in C_0^\infty(\Omega)$ and $\nu$ is the outward normal to the boundary of $\Omega$. We assume that $\Omega\subseteq\R^n$ is a domain of $\R^n$ with smooth boundary. When $\Omega$ is bounded, the well-posedness of \eqref{fp} in parabolic energy spaces are standard matter that can be found in \cite[Chapter III]{LSU}. Moreover, in this setting the conservation of mass follows easily by using the test function identically equal to 1. \\
 In this paper we work under the main assumption that $\Omega$ is unbounded and satisfies the following condition
\begin{align}\label{intcone}
&\Omega\mbox{ satisfies the interior cone condition, i.e. there exists a finite cone such that every}\\ &\mbox{ point in $\Omega$ is the vertex of a cone (congruent to the fixed given cone) contained in $\Omega$.}\notag
\end{align}
This geometric property ensures suitable extension properties of Sobolev spaces and a priori bounds for parabolic equations with divergence-type terms that give the uniqueness of solutions and the conservation of mass property, see respectively \cite[Sections 7-8]{Daners} and \cite{GoffiTralli} (see p.11 for the appropriate bounds needed to run the calculations).
All these results were obtained under the main assumption that 
\[
b\in L^\infty(Q_\tau).
\]
The regularity condition on $b$ can be considerably weakened, as discussed in \cite[Remark 2.5]{GoffiTralli}: for instance, one can assume the local Aronson-Serrin type condition $b\in L^Q_t(L^P_x)$ with $\frac{n}{2P}+\frac{1}{Q}\leq\frac12$ combined with suitable conditions at infinity.
\begin{thm}\label{propGT}[Theorem 2.4 in \cite{GoffiTralli}]
Assume that $\Omega$ satisfies \eqref{intcone}, and let $b\in L^\infty(Q_\tau)$. Then there exists a unique weak energy solution $\rho_\eps\in W$  to \eqref{fp}, where
\[
W:=\left\{f\in L^2(\tau,T;W^{1,2}(\Omega))\,\mbox{ such that }\,\partial_t f\in L^2(\tau,T;W'_\Omega)\right\},\]
$W'_\Omega$ being the dual space of $W^{1,2}(\Omega)$, endowed with the standard dual norm. In addition, if $\rho_\tau\geq0$, then $\rho_\eps(t)\geq0$ and we have
$$
\int_{\Omega}\rho_\eps(x,t)\,dx=\int_\Omega\rho_\tau(x)\,dx
$$
for $t\in(\tau,T]$.
\end{thm}
\section{Rate of convergence estimates}\label{sec;rate}
In this section we consider domains with smooth boundary satisfying the following additional constraint
\begin{equation}\label{conv}
\Omega\mbox{ is convex.}
\end{equation}
We address the vanishing viscosity approximation of the following equation
\begin{equation}\label{HJ}
\begin{cases}
-\partial_t u+H(Du(x,t))=f(x,t)&\text{ in }Q_T:=\Omega\times(0,T),\\
\partial_\nu u=0& \text{ in }\Sigma_T:=\partial\Omega\times(0,T),\\
u(x,T)=u_T(x)&\text{ in }\Omega,
\end{cases}
\end{equation}
via the diffusive Cauchy-Neumann problem
\begin{equation}\label{HJvisc}
\begin{cases}
-\partial_t u_\eps-\eps\Delta u_\eps+H(Du_\eps(x,t))=f(x,t)&\text{ in }Q_T,\\
\partial_\nu u_\eps=0& \text{ in }\Sigma_T,\\
u_\eps(x,T)=u_T(x)&\text{ in }\Omega,
\end{cases}
\end{equation}
where $\Omega$ is either bounded and satisfies \eqref{conv} or it is unbounded and satisfies \eqref{intcone} and \eqref{conv}. Note that if the domain is bounded the assumption \eqref{intcone} can be dropped, as it is a consequence of the convexity of the ambient space, cf. \cite[Section 2.1]{GoffiTralli}.  In the case of \eqref{HJ}, since $u$ is not $C^1$, the Neumann condition and the notion of solution should be interpreted in the viscosity sense, see \cite{LionsDuke}. 
In this setting, the notion of viscosity solution and the well-posedness results also hold in the case of unbounded domains \cite[p. 795 and p. 806]{LionsDuke}, see also \cite[Remark p. 804]{LionsDuke} for additional comments about condition \eqref{conv}. In the sequel we will exploit the notation $u(t)=u(\cdot,t)$.\\

All the results in the paper will be stated in the form of a priori estimates and we will assume that our problem admits classical solutions $u_\eps \in C^{2,1}$. However, one needs some qualitative properties of solutions to perform rigorously the calculations below, see Remark \ref{details} for more details.

\begin{thm}\label{main1}
Let $u_\eps$ be a Lipschitz solution to the viscous equation \eqref{HJvisc} and $u$ be a Lipschitz viscosity solution to the first-order equation \eqref{HJ}, with $\Omega$ either bounded and satisfying \eqref{conv} or unbounded and satisfying \eqref{intcone} and \eqref{conv}. Assume that $H\in W^{1,\infty}_{\mathrm{loc}}(\R^n)$ and $f\in W^{1,\infty}(Q_T)$. Then
\[
\|u_\eps-u\|_{L^\infty(Q_T)}\leq M\sqrt{\eps T}.
\]
for a positive constant $M$ depending on $\|Du\|_{L^\infty(\overline{Q}_T)},n,\|Df\|_{L^\infty(Q_T)}$.
\end{thm}

\begin{proof}
We differentiate \eqref{HJvisc} with respect to $\eps$ to find the PDE
\[
-\partial_t v^\eps-\eps\Delta v^\eps+D_pH(Du_\eps)\cdot Dv^\eps=\Delta u_\eps.
\]
equipped with the boundary conditions
\[
\partial_\nu v^\eps=0\text{ on }\Sigma_T\text{ and }v^\eps(T)=0\text{ in }\Omega.
\]
This procedure of differentiating with respect to $\eps$ can be made rigorous via Lemma 2.1 in \cite{CirantGoffiIndiana}. We introduce the adjoint problem
\begin{equation}\label{adjoint}
\begin{cases}
\partial_t\rho_\eps-\eps\Delta\rho_\eps-\mathrm{div}(D_pH(Du_\eps)\rho_\eps)=0&\text{ in }Q_\tau=\Omega\times(\tau,T)\\
\eps \partial_\nu \rho_\eps+\rho_\eps D_pH(Du_\eps)\cdot \nu=0&\text{ on }\Sigma_\tau,\\
\rho_\eps(x,\tau)=\delta_{x_0}&\text{ in }\Omega.
\end{cases}
\end{equation}
To make the next procedure rigorous, we need to consider $\rho(\tau)\in C_0^\infty(\Omega)$ satisfying $\|\rho(\tau)\|_1=1$, $\rho(\tau)\geq0$ and finally pass to the supremum over $\rho(\tau)$.
Since $H$ is locally Lipschitz and $u$ is globally Lipschitz, the velocity field $b(x,t)=-D_pH(Du_\eps)\in L^\infty(Q_\tau)$. Thus, \eqref{adjoint} admits a unique weak nonnegative solution belonging to $W$ by Theorem \ref{propGT}. By duality, this gives the representation formula
\begin{equation}\label{repr}
\int_\Omega v^\eps(\tau)\rho_\eps(\tau)\,dx=\int_\Omega v^\eps(T)\rho_\eps(T)\,dx+\iint_{Q_\tau}\Delta u_\eps\rho_\eps\,dxdt.
\end{equation}
We only need to bound the last integral of the right-hand side, since the first one vanishes due to the fact that $v^\eps(T)=0$. Standard computations through the Bochner's identity yield the evolution of $w_\eps=|Du_\eps|^2$
\[
-\partial_t w_\eps-\eps\Delta w_\eps+2\eps|D^2u_\eps|^2+D_pH(Du_\eps)\cdot Dw_\eps=2Df\cdot Du_\eps.
\]
The convexity of $\Omega$, instead, provides the inequality $\partial_\nu w_\eps\leq0$ on $\Sigma_\tau$, cf. \cite[Lemma 2.1]{GoffiTralli}. 
By duality with the adjoint problem \eqref{adjoint} and using that $\int_{\partial\Omega}\partial_\nu w_\eps\rho_\eps\leq0$, we have the inequality
\[
\int_{\Omega} w_\eps(\tau)\rho_\eps(\tau)\,dx+\int_\tau^T\int_\Omega 2\eps|D^2u_\eps|^2\rho_\eps\,dxdt\leq\int_{\Omega} w_\eps(T)\rho_\eps(T)\,dx+\int_\tau^T\int_\Omega 2Df\cdot Du_\eps\rho_\eps\,dxdt.
\]
By the Young's inequality and the conservation of mass we estimate
\[
\int_\tau^T\int_\Omega 2Df\cdot Du_\eps\rho_\eps\,dxdt\leq 2\|Df\|_{L^{\infty}(Q_\tau)}\|Du_\eps\|_{L^\infty(Q_\tau)}(T-\tau)\leq \frac{\|Du_\eps\|_{L^\infty(Q_\tau)}^2}{2}+2\|Df\|_{L^{\infty}(Q_\tau)}^2T^2.
\]
Using again the conservation of mass property in Theorem \ref{propGT} and recalling that $w_\eps=|Du_\eps|^2$, we deduce the following bound
\begin{equation}\label{lip}
|Du_\eps(x,\tau)|^2+2\eps\int_\tau^T\int_\Omega |D^2u_\eps|^2\rho_\eps\,dxdt\leq C_L
\end{equation}
for a constant $C_L$ depending on $\|Du_T\|_{L^\infty(\Omega)}, \|Df\|_{L^\infty(Q_\tau)}$. Note that this step requires the convexity of the ambient space, since we use that $\partial_\nu w_\eps\leq0$, $\partial_\nu u_\eps=\partial_\nu \rho_\eps=0$ on $\Sigma_T$, along with the nonnegativity of $\rho_\eps$ on $Q_\tau$. Therefore, we can use \eqref{lip} and Theorem \ref{propGT} to bound by the H\"older's inequality the right-hand side of \eqref{repr}:
\[
\iint_{Q_\tau}\Delta u_\eps\rho_\eps\,dxdt\leq \sqrt{n}\left(\iint_{Q_\tau}|D^2 u_\eps|^2\rho_\eps\right)^\frac12\left(\iint_{Q_\tau}\rho_\eps\right)^\frac12\leq \frac{\sqrt{C_L nT}}{\sqrt{\eps}}.
\]
Integrating in $\eps$ and noting that the right-hand side is integrable near $\eps=0$, we get for $\eps_1\geq \eps_2>0$
\[
\|(u^{\eps_1}-u^{\eps_2})(\tau)\|_{L^\infty(\Omega)}\leq M\sqrt{T}(\sqrt{\eps_1}-\sqrt{\eps_2}),
\]
where $M=2\sqrt{nC_L}$. We conclude by letting $\eps_2\to0$.
\end{proof}

\begin{rem}
In the case of bounded domains, the same result can be achieved using the standard maximum principle for parabolic equations applied to the auxiliary function \[F_\pm(x,t)=\sqrt{\eps}\partial_\eps u_\eps(x,t)\pm|Du_\eps(x,t)|^2.\] Our proof provides a global approach to this problem that avoids the use maximum principles on unbounded spaces.
\end{rem}

\begin{rem}\label{details}
Solutions $u_\eps$ of viscous Hamilton-Jacobi equations $-\partial_t u_\eps-\eps\Delta u_\eps+H(Du_\eps)=f$, $f\in W^{1,\infty}$, satisfy the following bounds independent of the viscosity
\[
\|\partial_t u_\eps\|_{L^\infty},\|u_\eps\|_{L^\infty},\|Du_\eps\|_{L^\infty}<\infty.
\]
Moreover, $u_\eps\in C^{1,1}$ independent of $\eps$ for short-time horizons or small initial data, at least for domains without boundary \cite[Proposition 12.1]{L82book}. We also have the following first- and second-order regularizing effects independent of $\eps$
\[
\|Du_\eps\|_{L^\infty}\lesssim \frac{1}{\sqrt{t}}\text{ and }\Delta u_\eps(t)\lesssim \frac{1}{t}.
\]
Estimates on space derivatives require either coercivity and/or convexity conditions or linear growth conditions on $H$; see \cite{BD,Souplet,L82book,LionsDuke}. 
 Once Lipschitz estimates are established, one can reach classical smoothness via maximal regularity properties for the heat equation: this can be done via bootstrap arguments or directly applying maximal $L^q$- regularity properties for Hamilton-Jacobi equations. The latter hold when either $H$ has at most linear growth or it has (superlinear) power-growth, see \cite{Gofficcm,GoffiPediconi} (the convexity of $\Omega$ is not needed in general when $H$ has at most natural growth). Classical regularity of solutions depends on the diffusion and hence on the viscosity $\eps$.\\
 If the domain is unbounded, we need some extra conditions on the geometry of the domain at infinity to ensure enough regularity solutions to justify the calculations. For instance, one can assume that $\Omega$ is uniformly regular of class $C^k$: this condition is stronger than \eqref{intcone}, but it represents only a qualitative assumption since it is not exploited in the derivation of our estimates. See \cite[Corollary 15.7]{Amann}, \cite{Browder} and \cite[Section 2.1]{GoffiTralli} for a complete discussion.
Examples of Lipschitz bounds in the unbounded case can be found in \cite{GoffiTralli} for quadratic Hamilton-Jacobi equations (note that they are obtained for log-solutions of the heat equation). We also refer to the procedure that leads to the second-order integral bound \eqref{lip}, which hides an a priori Lipschitz estimate independent of the viscosity. At this stage, we do not know how to remove the convexity condition on $\Omega$ to obtain the second order integral bound in \eqref{lip}. However, Lipschitz estimates are available under weaker geometric assumptions on the domain, see e.g. \cite[p. 807]{LionsDuke}.
\end{rem}

We now investigate an improvement of the convergence rate in Theorem \ref{main1}. The authors in \cite[p. 18]{PerthameSanders} highlighted that for a certain one-dimensional, quadratic, Hamilton-Jacobi equation with Neumann conditions, one should expect a rate of convergence better than the order $\mathcal{O}(\sqrt{\eps})$. Our next result confirms the observations of \cite{PerthameSanders}, as it provides an improvement when the solutions are semi-superharmonic and $H$ is quadratic, under additional conditions on $f$, cf. \eqref{condf}. In particular, it gives a quantitative statement of \cite[Theorem 8.1]{L82book}. If $f$ is semi-superharmonic and $H$ is uniformly convex, this upgrade is known under certain conditions of the state space: the paper \cite{CirantGoffiIndiana} provides two-sides $\mathcal{O}(\eps|\log \eps|)$ rates of convergence in the periodic setting, see also \cite{ChaintronDaudin} for the whole space case. Both of them exploit entropy bounds of solutions to the adjoint problem \eqref{adjoint}. The rate can be also improved under smallness conditions (e.g. short time horizons) or under geometric requirements on the solutions, at least for domains without boundary, cf. Remark 3.3 in \cite{CirantGoffiIndiana}.

\begin{thm}\label{main2}
Let $u_\eps\in W^{1,\infty}_x$ be a semi-superharmonic solution (i.e. satisfying the one-side bound $\Delta u_\eps\leq c(t),\ c\in L^1(0,T)$ independent of $\eps$) to the viscous equation \eqref{HJvisc} and $u\in W^{1,\infty}_x$ be a semi-superharmonic solution to the first-order equation \eqref{HJ}, with $\Omega$ as above. Assume that $H(p)=|p|^2$ and \begin{equation}\label{condf}\Delta f\leq c_f\text{ in $Q_T$, and }\partial_\nu f\geq0\text{ on }\Sigma_T.\end{equation} Then, for all $\beta\in(1/2,1)$, there exists $C=C(\beta)$ depending in addition on $n,c_f,T,\|Du\|_{L^\infty(\overline{Q}_T)}$,\\$\|Df\|_{L^\infty(Q_T)}$, $\|(\Delta u_T)^+\|_{L^\infty(\Omega)}$ such that
\[
-C\eps^\beta\leq u_\eps-u\leq \eps\int_0^Tc(t)\,dt,\ \beta\in(1/2,1).
\]
\end{thm}

\begin{rem}
Some observations about \eqref{condf} are in order. While the one-side condition $\Delta f\leq c_f$ is natural to derive unilateral second order bounds, the second condition on the boundary is in general necessary. We refer to \cite[Remark 8.2]{L82book} for some explicit examples in bounded and unbounded convex domains showing that this condition cannot be dropped. Moreover, the condition $\Delta u_\eps\leq c(t)$ hides the condition $\Delta u_T\leq M_0$ and it will be addressed in Theorem \ref{sshbounds}.
\end{rem}

\begin{proof}
We start again with the equation solved by $v^\eps$, namely
\[
-\partial_t v^\eps-\eps\Delta v^\eps+D_pH(Du_\eps)\cdot Dv^\eps=\Delta u_\eps
\]
equipped with the boundary conditions
\[
\partial_\nu v^\eps=0\text{ on }\Sigma_T\text{ and }v^\eps(T)=0\text{ in }\Omega.
\]
By duality obtained by using the test function $\rho_\eps$ solving \eqref{adjoint} in the weak formulation of $v^\eps$, we obtain the representation formula
\[
\int_\Omega v^\eps(\tau)\rho_\eps(\tau)\,dx=\int_\Omega v^\eps(T)\rho_\eps(T)\,dx+\iint_{Q_\tau}\Delta u_\eps\rho_\eps\,dxdt.
\]
In view of the standing assumptions $\Delta u_\eps\leq c(t)$ and $v^\eps(T)=0$. By the conservation of mass for the adjoint problem in Theorem \ref{propGT} and using that $\rho_\eps\geq0$, we have
\[
\iint_{Q_\tau}\Delta u_\eps\rho_\eps\,dxdt\leq \int_\tau^T\|(\Delta u_\eps(t))^+\|_{L^\infty(\Omega)}\int_\Omega\rho_\eps\,dx\,dt\leq \int_0^Tc(t)\,dt.
\]
Passing to the supremum over $\rho(\tau)\in C_0^\infty(\Omega)$ satisfying $\|\rho(\tau)\|_1=1$, $\rho(\tau)\geq0$ on the integral $\int_\Omega v^\eps(\tau)\rho_\eps(\tau)\,dx$, and then integrating with respect to $\eps$, we conclude
\[
\|(u_\eps-u)^+(\tau)\|_{L^\infty(\Omega)}\leq \eps\int_0^Tc(t)\,dt.
\]
We now prove the leftmost estimate. We first show that for $\alpha\in(1,2)$ we have the bound
\begin{equation}\label{secondord}
\int_\tau^T\int_{\Omega}(t-\tau)^\alpha|D^2u_\eps|^2\rho_\eps\, dxdt\leq K,
\end{equation}
where $\rho_\eps$ solves \eqref{adjoint} with $b(x,t)=-D_pH(Du_\eps)=-2Du_\eps$ and $K$ depends on $n, \alpha, T,u_T,f$. To this aim, we exploit the uniform convexity of the Hamiltonian. We first find, by differentiating twice the equation for $u_\eps$ and using the Bochner's identity $\Delta |Du_\eps|^2=2|D^2u_\eps|^2+2Du_\eps\cdot D\Delta u_\eps$, the following inequality solved by the function $z_\eps(x,t)=(t-\tau)^\alpha  \Delta u_\eps(x,t)$
\[
-\partial_t z_\eps-\eps\Delta z_\eps+2 (t-\tau)^\alpha  |D^2u_\eps|^2 +2Du_\eps\cdot Dz_\eps
= -\alpha (t-\tau)^{\alpha-1}\Delta u_\eps +(t-\tau)^\alpha\Delta f \quad \text{ in }Q_T.
\]
Note that since $u_\eps$ solves the Hamilton-Jacobi equation, $z$ satisfies the boundary condition
\[
\partial_\nu z_\eps=(t-\tau)^\alpha\partial_\nu  \Delta u_\eps(x,t)=(t-\tau)^\alpha\partial_\nu \left\{\frac1\eps\left(-\partial_t u_\eps+|Du_\eps|^2 -f\right)\right\}\leq0\text{ on }\Sigma_T.
\]
To see the last inequality, use the identity $\partial_\nu(\partial_t u_\eps)=0$, $\partial_\nu |Du_\eps|^2\leq0$ because $\partial_\nu u_\eps=0$ and $\Omega$ is convex, and also \eqref{condf}.
By duality and integrating in $\Omega\times(\tau,T)$, exploiting also that $\int_{\partial\Omega}\partial_\nu z_\eps \rho_\eps\leq0$, we have
\begin{align*}\label{secondorde}
\underbrace{\int_{\Omega}z_\eps(\tau)\rho_\eps(\tau)\,dx}_{=0}&+2\int_\tau^T\int_{\Omega}(t-\tau)^\alpha |D^2u_\eps|^2 \rho_\eps\,dxdt\leq \underbrace{\int_{\Omega}z_\eps(T)\rho_\eps(T)\,dx}_{\leq (T-\tau)^\alpha\|(\Delta u_T)^+\|_{L^\infty(\Omega)}}\\
&-\alpha \int_\tau^T\int_{\Omega} (t-\tau)^{\alpha-1}\Delta u_\eps \rho_\eps\,dxdt+\int_\tau^T\int_{\T^n} (t-\tau)^{\alpha} \Delta f \rho_\eps\,dxdt.
\end{align*}
Note that to integrate by parts we used the boundary conditions $\partial_\nu z_\eps\leq0$, $\partial_\nu \rho_\eps=\partial_\nu u_\eps=0$ on $\Sigma_T$, together with the fact that $\rho_\eps\geq0$ by Theorem \ref{propGT}. We now use the Young's inequality and Theorem \ref{propGT} as follows
\begin{align*}
-\alpha &\int_\tau^T\int_{\Omega} (t-\tau)^{\alpha-1} \Delta u_\eps \rho_\eps\,dxdt\\
&\leq \frac{1}{n}\int_\tau^T\int_{\Omega}(t-\tau)^\alpha|\Delta u_\eps |^2\rho_\eps\,dxdt+\frac{n\alpha^2}{4}\int_\tau^T\int_{\Omega}(t-\tau)^{\alpha-2}\rho_\eps\,dxdt\\
&\leq \int_\tau^T\int_{\Omega}(t-\tau)^\alpha|D^2u_\eps|^2\rho_\eps\,dxdt+\frac{n\alpha^2}{4}\int_\tau^T (t-\tau)^{\alpha-2}\,dt\\
&\leq \int_\tau^T\int_{\Omega}(t-\tau)^\alpha|D^2u_\eps|^2\rho_\eps\,dxdt+\frac{n\alpha^2}{4(\alpha-1)} T^{\alpha-1}.
\end{align*}
We then obtain
\begin{equation*}\label{est1}
\int_\tau^T\int_{\Omega}(t-\tau)^\alpha|D^2u_\eps|^2\rho_\eps\,dxdt
\leq \frac{n \alpha^2}{4(\alpha-1)}T^{\alpha-1}+T^\alpha\|(\Delta u_T)^+\|_{L^\infty(\Omega)}+\frac{T^{\alpha+1}}{\alpha+1}c_f=:K.
\end{equation*}
\medskip
We now conclude the estimate. Assume first that $\tau + \eps < T$. Then, back to the representation formula \eqref{repr}, we have
\begin{align*}
\left| \int_\tau^T\int_{\Omega}\Delta u_\eps\rho_\eps\,dxdt \right| &\leq \sqrt{n}\left(\int_{\tau + \eps}^{T}\int_{\Omega}(t-\tau)^{\alpha/2}|D^2u_\eps|\rho_\eps (t-\tau)^{-\alpha/2}\,dxdt+\int_{\tau}^{\tau + \eps}\int_{\Omega}|D^2u_\eps|\rho_\eps\,dxdt \right)\\
&\le \sqrt{n}\left(\int_{\tau + \eps}^{T}\int_{\Omega}(t-\tau)^{\alpha}|D^2 u_\eps|^2\rho_\eps\,dxdt \right)^\frac12\left(\int_{\tau + \eps}^{T}\int_{\Omega} (t-\tau)^{-\alpha}\rho_\eps\,dxdt\right)^\frac12\\
& \quad + \sqrt{n} \left(\int_{\tau}^{\tau + \eps}\int_{\Omega}|D^2u_\eps|^2\rho_\eps\,dxdt \right)^\frac12\left(\int_{\tau}^{\tau + \eps}\int_{\Omega}\rho_\eps\,dxdt  \right)^\frac12\\
&\leq \sqrt{\frac{nK\eps^{1-\alpha}}{\alpha-1}} + \sqrt {nC_L},
\end{align*}
where we used the estimate \eqref{est1}, the conservation of mass of the adjoint problem and \eqref{lip}. If $\tau \ge T-\eps$, then the same estimate holds, without the first constant $\sqrt{\frac{nK\eps^{1-\alpha}}{\alpha-1}}$, since in this case there is no need to split the first integral. We conclude that
\[
v^\eps(x_0, \tau) \ge - \sqrt{\frac{nK}{\alpha-1}} \eps^{\frac{1-\alpha}2} - \sqrt {nC_L}.
\]
Since the previous estimate holds for all $x_0, \tau$, we obtain, by setting $\frac{3-\alpha}2=\beta\in(1/2,1)$, 
\[
u_{\eps}-u\geq - \frac1\beta\sqrt{\frac{nK}{2(1-\beta)}} \eps^{\beta} - \sqrt {nC_L} \eps.
\]

\end{proof}

\begin{rem}
It would be worth investigating the improvement of the rate of convergence in weaker $L^1$ norms, as in \cite{CGM,LinTadmor}. In this case, starting from the identity 
\[
\int_\Omega v^\eps(\tau)\rho_\eps(\tau)\,dx=\int_\Omega v^\eps(T)\rho_\eps(T)\,dx+\iint_{Q_\tau}\Delta u_\eps\rho_\eps\,dxdt,
\]
we can estimate
\[
\iint_{Q_\tau}\Delta u_\eps\rho_\eps\,dxdt\leq \|\Delta u_\eps\|_{L^1(Q_\tau)}\|\rho_\eps\|_{L^\infty(Q_\tau)}.
\]
If the right-hand side does not depend on $\eps$ we would obtain a $\mathcal{O}(\eps)$ rate of convergence: this is known by the results in \cite{LinTadmor,CGM}. The estimate $\|\Delta u_\eps\|_{L^1(Q_\tau)}$ can be obtained on bounded convex domains \cite{L82book}, while one needs maximum principle estimates for the adjoint problem independent of the viscosity to bound $\|\rho_\eps\|_{L^\infty(Q_\tau)}$. This holds in general under the assumption that $[\mathrm{div}(b)]^-<\infty$, which is satisfied in the setting of Hamilton-Jacobi equations. For instance, if $H(Du_\eps)=|Du_\eps|^2$ we have $b=-D_pH(Du_\eps)=-2Du_\eps$ and hence
\[
-\mathrm{div}(b)=2\Delta u_\eps.
\]
Therefore, if $u_\eps$ is semi-superharmonic, the coefficient $b$ of the advection-diffusion PDE solved by $\rho_\eps$ satisfies $[\mathrm{div}(b)]^-<\infty$. A study in this direction was carried out in \cite[Lemma 3.4]{CDS} on unbounded domains when $\mathrm{div}(b)$ has two-side bounds.
\end{rem}

\begin{rem}[Towards the $\mathcal{O}(\eps|\log\eps|)$ rate of convergence] In the case of domains without boundary, recent results provided a better (two-side) $\mathcal{O}(\eps|\log\eps|)$ speed of convergence, cf. \cite{ChaintronDaudin,CirantGoffiIndiana}. These are based on entropy-type estimates for solutions to \eqref{adjoint} which are unknown in the context of unbounded Neumann problems. We do not know whether in the setting of this manuscript the speed of convergence can be improved. It is unclear whether this latter rate $\mathcal{O}(\eps|\log\eps|)$ can be boosted to a rate in $L^p$ norms, $p>1$, independent of $p$. 
\end{rem}

\section{One-side second derivative bounds for solutions of Hamilton-Jacobi equations}\label{sec;bounds}
Theorem \ref{main2} can be made unconditional using semi-superharmonic bounds for strictly convex Hamilton-Jacobi equations. We recall that under the uniform convexity condition on $H$, solutions are known to be semiconcave when $\Omega=\R^n$ or $\T^n$. However, this property remains an open question for Neumann problems posed on bounded convex domains. We show however that second order bounds on pure derivatives independent of the viscosity can be achieved for Hamilton-Jacobi Neumann boundary-value problems. The proof uses an integral approach based again on Theorem \ref{propGT} . Earlier results in this direction appeared in \cite{L82book} via maximum principle methods on bounded convex domains.\\

The first result we present shows the conservation of semi-superharmonic estimates (independent of the viscosity) from the initial datum for the Hamilton-Jacobi flow when $H=H(Du)$ has power-growth $\gamma>1$. To our knowledge this was known only for stationary equations with zero-th order coefficients posed on bounded convex domains \cite[Theorem 8.1]{L82book}.
\begin{thm}\label{sshbounds} Let $\Omega$ be an unbounded domain satisfying \eqref{intcone}-\eqref{conv} (or bounded and satisfying \eqref{conv}), and assume that $\|(\Delta u_T)^+\|_{L^\infty(\Omega)}\leq M_0$, with $H$ satisfying
\[H(p)=|p|^\gamma,\ \gamma>1.\] 
Assume also
\[
\Delta f\leq c_f(t)\in L^1(0,T)\text{ in $Q_T$ and }\partial_\nu f\geq0\text{ on }\Sigma_T.
\]
Then, any Lipschitz solution of \eqref{HJvisc} satisfies
\[
\|(\Delta u)^+(t)\|_{L^\infty(\Omega)}\leq M_0+\int_0^Tc_f(t)\,dt.
\]
Note that the estimate does not depend on $\eps$.
\end{thm}
\begin{proof}[Sketch of the proof]
We follow the proof of the Li-Yau estimate in \cite{GoffiTralli}. When $H(p)=|p|^2$, it is enough to observe that $\tilde z(x,t)=\Delta u_\eps$ satisfies the inequality
\begin{equation}\label{z}
-\partial_t \tilde z_\eps-\eps\Delta \tilde z_\eps+2|D^2u_\eps|^2+D_pH(Du_\eps)\cdot D\tilde z_\eps= \Delta f\text{ in }Q_T
\end{equation}
with the boundary condition $\partial_\nu \tilde z\leq0$ on $\Sigma_T$, as in Theorem \ref{main2}. A similar calculation in the more general case $H(Du_\eps)=|Du_\eps|^\gamma$, or its smooth approximation $H_\delta(Du_\eps)=(\delta+|Du_\eps|^2)^{\frac{\gamma}{2}}$, leads to
\[
-\partial_t \tilde z_\eps-\eps\Delta \tilde z_\eps+\gamma|Du_\eps|^{\gamma-2}Du_\eps\cdot D\tilde z_\eps\leq \Delta f.
\]
In this case $\partial_\nu \tilde z_\eps\leq0$ on $\partial\Omega$, since $\partial_\nu |Du_\eps|^\gamma\leq0$ on $\partial\Omega$ because $\Omega$ is convex. By duality with the adjoint problem \eqref{adjoint} and using its conservation of mass, it is immediate to obtain the desired estimates, which are independent of the viscosity.
\end{proof}

The second result concerns a second order regularizing effect without assuming any regularity condition on the initial datum (continuity is in general necessary for the well-posedness of solutions in the viscosity sense). This is reminiscent of the uniqueness condition for first-order uniformly convex Hamilton-Jacobi equations introduced by A. Douglis \cite{Douglis} and the one-side Lipschitz uniqueness assumption for scalar conservation laws introduced by O. Oleinik. The main novelty here is the possibility of consider unbounded ambient spaces under the Neumann condition, though the case of bounded domains seems new.
\begin{thm}
Let $\Omega$ be an unbounded domain satisfying \eqref{intcone}-\eqref{conv} (or bounded and satisfying \eqref{conv}), and assume that $f=0$ with $H$ of the form
\[H(p)=|p|^2.\] 
Then, any Lipschitz solution of \eqref{HJvisc} satisfies
\[
\|(\Delta u)^+(t)\|_{L^\infty(\Omega)}\leq \frac{n}{2(T-t)}.
\]
Note that the estimate does not depend on $\eps$.
\end{thm}
\begin{proof}
From \eqref{z} we have that $\tilde z_\eps(x,t)=\Delta u_\eps(x,t)$ solves 
\[
-\partial_t \tilde z_\eps-\eps\Delta \tilde z_\eps+2|D^2u_\eps|^2+2Du_\eps\cdot D\tilde z_\eps=0
\]
Define now $\tilde h_\eps(x,t)=(T-t)^2z_\eps(x,t)$ and discover that it solves
\[
-\partial_t \tilde h_\eps-\eps\Delta \tilde h_\eps+2(T-t)^2|D^2u_\eps|^2+2Du_\eps\cdot D\tilde h_\eps=2(T-t)\Delta u_\eps.
\]
Moreover, $\partial_\nu \tilde h_\eps\leq 0$ on $\Sigma_T$ since $\partial_\nu \tilde z_\eps\leq0$. Testing the previous equations against the solution $\rho_\eps$ of \eqref{adjoint} on $\Omega\times(\tau,T)$ and noting that $\tilde h_\eps(T)=0$, we conclude
\[
\int_\Omega \tilde h_\eps(\tau)\rho_\eps(\tau)\,dx+2\int_\tau^T\int_\Omega (T-t)^2|D^2u_\eps|^2\rho_\eps\,dxdt\leq2\int_\tau^T\int_\Omega (T-t) \Delta u_\eps\rho_\eps\,dxdt.
\]
We now bound the last integral. We have by Cauchy-Schwarz and Young's inequalities
\begin{align*}
2\int_\tau^T\int_\Omega (T-t)\Delta u_\eps\rho_\eps\,dxdt&\leq 2\sqrt{n}\int_\tau^T\int_\Omega (T-t)|D^2 u_\eps|\rho_\eps\,dxdt\\
&\leq 2\int_\tau^T\int_\Omega (T-t)^2|D^2u_\eps|^2\rho_\eps\,dxdt+\frac{n}{2}\int_\tau^T\int_\Omega \rho_\eps\,dxdt\\
&\leq 2\int_\tau^T\int_\Omega (T-t)^2|D^2u_\eps|^2\rho_\eps\,dxdt+\frac{n}{2}\int_\tau^T 1 \,dt.
\end{align*}
Therefore
\[
(T-\tau)^2\int_\Omega \Delta u_\eps(\tau)\rho_\eps(\tau)\,dx=\int_\Omega \tilde h_\eps(\tau)\rho_\eps(\tau)\,dx\leq \frac{n}{2}(T-\tau),
\]
which implies the desired estimate passing to the supremum over $\rho_\tau\geq0$.
\end{proof}


\end{document}